\documentclass[reqno,12pt]{amsart}
\newcommand{\CC}{{\mathbb C}}

\def\bege{\begin{equation}} \def\ende{\end{equation}}
\def\begr{\begin{eqnarray}} \def\endr{\end{eqnarray}}
\usepackage{amsmath}
\usepackage{amssymb}
\usepackage{cite}

\usepackage{bbm}
\usepackage{amsmath}
\usepackage{txfonts}
\usepackage{stmaryrd}
\usepackage{amssymb}
\usepackage{amsfonts}
\usepackage{times}
\usepackage{mathrsfs}

\usepackage{amsthm}
\usepackage{enumerate}

\usepackage{framed}
\usepackage{lipsum}
\usepackage{color}

\usepackage[colorlinks, citecolor=blue,pagebackref,hypertexnames=false]{hyperref}
\usepackage{pgf,tikz}

\newtheorem{define}{\hspace{0em}Definition}
\newtheorem{lemma}{\hspace{0em}Lemma}

\def\BB{ \mathbb{B}}
\def\SS{ \mathbb{S}}
\def\CC{ \mathbb{C}}
\newcommand{\DD}{{\mathbb D}}

\def\om{\omega}

\def\begr{\begin{eqnarray}} \def\endr{\end{eqnarray}}

\newtheorem{Theorem}{Theorem}

\newtheorem{Proposition}{Proposition}
\newtheorem{Remark}{Remark}

\begin{document}
\title[The essential norm of Toeplitz operators]{The essential norm of Toeplitz operators between Bergman spaces induced by doubling weights}
\author{Peiying Huang and Guangfu Cao$\dagger$}
\address{Peiying Huang\\School of Mathematics and Information Science, Guangzhou University, Guangzhou 510006, China.}

\address{Guangfu Cao\\School of Mathematics and Information Science, Guangzhou University, Guangzhou 510006, China.}

\email{hwangpuiying@163.com}
\email{guangfucao@163.com}


\begin{abstract} 
This paper investigates the essential norm of Toeplitz operators $\mathcal{T}_\mu$ acting from the Bergman space $A_\omega^p$ to $A_\omega^q$ ($1 < p \leq q < \infty$) on the unit ball, where $\mu$ is a positive Borel measure and $\omega \in \mathcal{D}$ (a class of doubling weights). Leveraging the geometric properties of Carleson blocks and the structure of radial doubling weights, we establish sharp estimates for the essential norm in terms of the asymptotic behavior of $\mu$ near the boundary. As a consequence, we resolve the boundedness-to-compactness transition for these operators when $1 < q < p<\infty$, showing that the essential norm vanishes exactly. 
These results generalize classical theorems for the unweighted Bergman space ($\omega \equiv 1$) and provide a unified framework for studying Toeplitz operators under both radial and non-radial doubling weights in higher-dimensional settings.

\thanks{$\dagger$ Corresponding author.}
\end{abstract}
\maketitle

\section{Introduction}

Let $\BB$ be the open unit ball of $\CC^n$ and $\SS$ denote its boundary. Let \(\omega\) be a weight, that is, a non-negative, measurable and integrable function, defined on \([0,1)\). For \(1 \leq p < \infty\), the radially weighted Bergman spaces $A^p_\omega$ consists of holomorphic functions $f \in H(\mathbb{B})$ satisfying 
$$
\|f\|_{A^p_\omega}^p := \int_{\mathbb{B}} |f(z)|^p \omega(z) \, dV(z) < \infty,
$$
where \(dV\) is the normalized volume measure on \(\mathbb{B}\). 
Notably, when $p=2$, $A^2_\omega$ becomes a Hilbert space with its canonical inner product \(\langle\cdot,\cdot\rangle_{A^2_\omega}\) defined through the standard integral representation. Recall that the \textit{reproducing kernel} 
$$
B^\omega_z(\xi) = \frac{1}{2n!} \sum_{\gamma\in\mathbb{N}^{n}} \frac{(n-1+|\gamma|)!}{\gamma! \, \omega_{2n+2|\gamma|-1}} \langle \xi,z \rangle^{\gamma}, 
$$
where \(\langle \cdot, \cdot \rangle\) is the standard inner product on \(\mathbb{C}^n\), and \(\omega_k := \int_0^1 r^k \omega(r) \, dr\).

The \textit{Toeplitz operator} \(\mathcal{T}_\mu\) with symbol \(\mu\), a given positive Borel measure, is defined by
\[\mathcal{T}_\mu f(z) := \int_{\mathbb{B}} f(\xi) \overline{B_z^\omega(\xi)} \, d\mu(\xi).\]
Luecking \cite{Luecking-1987} introduced Toeplitz operators $\mathcal{T}_\mu$ with measures as symbols and characterized their Schatten class membership on the unit disk $\DD$. These operators play a pivotal role across diverse disciplines, with foundational applications spanning mathematical physics (e.g.\cite{Ali-2005,Antonio-2016}), probability, complex analysis and operator theory (e.g.\cite{Coburn-1974,McDonald-1979,Pau-2015,Zhu-1990}).

In this paper, we will restrict our attention to a special class of weights called \textit{(radially) doubling weights}, defined as follows.
\begin{define}\label{def:doubling}
Let \(\omega\) be a weight on \([0,1)\) and \(\hat{\omega}(r) := \int_r^1 \omega(t) \, dt\), \(r \in [0,1)\). We say that:
\begin{itemize}
\item \(\omega \in \hat{\mathcal{D}}\), if there exists \(C \geq 1\) such that 
\[\hat{\omega}(r) \leq C \hat{\omega}\left( \frac{1 + r}{2} \right), \quad \forall \, 0 \leq r < 1;\]
\item $\omega \in \check{\mathcal{D}}$, if there exist \(C, K > 1\) such that 
\begin{align}\label{1.1}
\hat{\omega}(r) \geq C \hat{\omega}\left(1 - \frac{1 - r}{K}\right), \quad \forall \,0 \leq r < 1.
\end{align}
\end{itemize}
Denote
\[K_\omega := \inf \left\{ K : K \text{ satisfies the above condition } \eqref{1.1} \right\}.\]
We define \(\mathcal{D} := \hat{\mathcal{D}} \cap \check{\mathcal{D}}\).
\end{define}

A radial and continuous weight \(\omega\) is regular if there exists \(C > 0\) and \(\delta \in (0,1)\) such that 
\[\frac{1}{C} < \frac{\hat{\omega}(r)}{(1 - r)\omega(r)} < C, \quad \text{for } r \in (\delta, 1). \]
We denote the collection of regular radial weights by \(\mathcal{R}\). It is easy to see that $\mathcal{R}\subseteq\mathcal{D}$. The weight class $\hat{\mathcal{D}}$ was introduced by Pel\'aez in \cite{Pja-2016}, where he systematically studied the weighted Bergman spaces $A^p_\omega(\DD)$. For further details on $\hat{\mathcal{D}}$, $\mathcal{D}$ and operator theory on weighted Bergman spaces induced by these classes, we refer readers to \cite{Pja-2016,PjaRj-2014-book,PjaRj2015,PjaRj2016,PjRj2016jmpa,PjRj2019arxiv,PjaRjSk2015mz,PjaRjSk2018jga}.

For any \( z \in \mathbb{B} \setminus \{0\} \), let \( P_z \) be the orthogonal projection of \( \mathbb{C}^n \) onto the one-dimensional subspace \([z]:=\{\lambda z : \lambda \in \mathbb{C}\}\) generated by \( z \), and let \( P_{z}^{\perp} \) be the orthogonal projection from \( \mathbb{C}^n \) onto \( \mathbb{C}^n \ominus [z] \). Explicitly,
\[
P_z(w) = \frac{\langle w, z \rangle}{|z|^2} z, \quad P_{z}^{\perp}(w) = w - \frac{\langle w, z \rangle}{|z|^2} z, \quad w \in \mathbb{B}.
\]
For \( z = 0 \), we define \( P_0(w) = 0 \) and \( P_0^{\perp}(w) = w \). The pseudo-hyperbolic distance between $z$, $w\in\mathbb{B}$ is given by
\[
\rho(z,w) = \left| \frac{z - P_z(w) - \sqrt{1 - |z|^2} P_{z}^{\perp}(w)}{1 - \langle w, z \rangle} \right|.
\]
The Bergman metric \( \beta(\cdot, \cdot) \) on \( \mathbb{B} \) is defined as
\[\beta(z,w) := \frac{1}{2} \log \frac{1 + \rho(z,w)}{1 - \rho(z,w)}, \quad z, w \in \mathbb{B},\]
and the Bergman metric ball centered at \( z \) with radius \( r > 0 \) is
$$D(z,r) := \left\{w \in \mathbb{B} : \beta(z,w) < r\right\}.$$

To be more explicit on our results, we recall the Carleson measure associated with the weight $\omega$. For any $\xi,\tau\in\overline{\BB}$, define the nonisotropic metric $d(\xi,\tau)=|1-\langle \xi,\tau\rangle|^\frac{1}{2}$. For $r\in(0,1)$ and $\zeta\in\SS$, let
$$Q(\zeta,r)=\{\tau\in \SS:d(\zeta,\tau)\leq r\}.$$
$Q(\zeta,r)$  is a ball in $\SS$. More information about $d(\cdot,\cdot)$ and $Q(\zeta,r)$ can be found in \cite{Rw1980,Zk2005}.

For $a\in\BB\backslash\{0\}$, define $Q_a=Q(\frac{a}{|a|},\sqrt{1-|a|})$ and the associated Carleson block
$$S_a=S(Q_{a})=\left\{z\in\BB:\frac{z}{|z|}\in Q_{a},|a|<|z|<1\right\}.$$
For $a=0$, set $Q_a=\SS$ and $S_a=\BB$.  
The measure of a set $E\subset\BB$ with respect to $\om$ is given by $\om(E)=\int_E \om(z)dV(z)$. Further details on Carleson block and related analysis can be found in \cite{DjLsLxSy2019arxiv,Rw1980,Zk2005}.

Throughout this paper, the letter $C$ denotes a constant whose value may vary between different occurrences but remains independent of relevant parameters.
The notation $A \lesssim B$ signifies that $A\leq CB$ for some positive constant $C$. Similarly, $A \simeq B$ means $A\lesssim B$ and $B\lesssim A$.
 
We summarize key properties of doubling weights below.

\begin{lemma}\label{hat-d-eq-cd}
Let \(\omega\) be a radial weight and define $\omega^*(z) := \int_{|z|}^1 \omega(s) \log \frac{s}{|z|} s \, ds$ for $z \in \mathbb{B} \setminus \{0\}\). Then the following conditions are equivalent:
\begin{itemize}
\item[(i)]\(\omega \in \hat{\mathcal{D}}\);
\item[(ii)]\(\omega^*(z) \simeq (1 - |z|)\hat{\omega}(z)\) as \(|z| \to 1\);
\item[(iii)]There exist $C=C(\omega)>0$ and \(\beta = \beta(\omega) > 0\) such that
\begin{align*}
	\hat{\omega} (r) \leq C \left( \frac{1 - r}{1 - t} \right)^{\boldsymbol{\beta}} \hat{\omega} (t), \quad 0 \leq r \leq t < 1.
\end{align*}
\end{itemize}
\end{lemma}

\begin{proof}
The assertions follow from \cite[Lemma A]{PjaRjSk2018jga}.
\end{proof}

\begin{lemma}(\cite[Lemma 2.2]{Djun2025TO})\label{lem:2.2} 
Let \(\omega \in \hat{\mathcal{D}}\). Then:
\begin{itemize}
\item[(i)] For any \(\alpha > -2\), \((1 - r)^\alpha \omega^*(r) \in \mathcal{R}\);
\item[(ii)] \(\omega(S_a) \simeq (1 - |a|)^n \hat{\omega}(a)\);
\item[(iii)] \(\hat{\omega}(z) \simeq \hat{\omega}(a)\), if \(1 - |z| \simeq 1 - |a|\).
\end{itemize}
\end{lemma}

\begin{lemma}\label{TO-lem2-3}
Let \(\omega \in \hat{\mathcal{D}}\). Then: 
\begin{itemize}
\item[(i)]For all \(z \in \mathbb{B}\), 
$
\|B_z^\omega\|_{H^\infty} \simeq \frac{1}{\omega(S_z)},
$
where \(H^\infty\) is the space of all bounded holomorphic functions on \(\mathbb{B}\).
\item[(ii)]For all \(z \in \mathbb{B}\) and $1<p<\infty$,
$\|B_z^\omega\|_{A_\omega^p} \simeq  \frac{1}{\omega(S_z)^{1-\frac{1}{p}}}.$
\item[(iii)]There exist constants \(C = C(\omega) > 0\) and \(\delta = \delta(\omega) \in (0,1)\) such that $|B_a^\omega(z)| \geq \frac{C}{\omega(S_a)}$, $z \in S_{a_\delta}$, $a \in \mathbb{B} \setminus \{0\}$,
where \(a_\delta = (1 - \delta(1 - |a|))\frac{a}{|a|}\).
\end{itemize}
\end{lemma}
	
\begin{proof}
The first statement	is a consequence of \cite[Lemma 2.3]{Djun2025TO}; the last two statements follow from \cite[Lemma 3.2]{Djun2025TO}.
\end{proof}

\begin{lemma}(\cite[Lemma 2.4]{Djun2025TO})\label{lem:2.4}
	Suppose that \(1 < p, q < \infty\) and \(\omega \in \hat{\mathcal{D}}\). If \(\mu\) is a positive Borel measure and \(\mathcal{T}_\mu : A_\omega^p \to A_\omega^q\) is bounded, then \(\mathcal{T}_\mu\) is compact if and only if for every sequence \(\{f_k\}\) bounded in \(A_\omega^p\) that converges to 0 uniformly on compact subsets of \(\mathbb{B}\), 
	\[\lim_{k \to \infty} \|\mathcal{T}_\mu f_k\|_{A_\omega^p} = 0.\]
\end{lemma}

\begin{lemma}(\cite[Lemma 4]{DjLsLxSy2019arxiv})\label{DjLsLxSy2019-lem4}
Let \( 0 < p < \infty \), \(\omega\in \hat{\mathcal{D}}\). When \( |rz| > \frac{1}{4} \),
\[
M_p^p(r, B_z^{\omega}) \approx \int_0^{r|z|} \frac{1}{\hat{\omega}(t)^p (1 - t)^{np - n + 1}} \, dt.
\]
\end{lemma}

\begin{lemma}\label{check-d-eq-cd}
Let \(\omega\) be a radial weight. Then \(\omega \in \check{\mathcal{D}}\) if and only if there exist $C=C(\omega)>0$ and \(\alpha = \alpha(\omega) > 0\) such that
\begin{align*}
\hat{\omega} (t) \leq C \left( \frac{1 - t}{1 - r} \right)^{\boldsymbol{\alpha}} \hat{\omega} (r), \quad 0 \leq r \leq t < 1.
\end{align*}
\end{lemma}

\begin{proof}
Adapt the proof of \cite[Lemma 2.1]{Pja-2016} with minor adjustments.
\end{proof}

\begin{lemma}(\cite[Lemma 4]{2020-hankel-JA})\label{hankel-bb-lem4}
Let \(\omega, \nu \in \mathcal{D}\), and denote $\sigma = \sigma_{\omega,\nu} = \omega \hat{\nu}/\hat{\omega}$.  
Then \(\hat{\sigma} \simeq \hat{\nu}\) on \([0,1)\), and consequently \(\sigma \in \mathcal{D}\).
\end{lemma}

\begin{lemma}\label{intergral-eq}
Let \(0 < p < \infty\), \(\omega \in \mathcal{D}\), and \(-\alpha < \kappa < \infty\), where \(\alpha = \alpha(\omega) > 0\) is as in Lemma \ref{check-d-eq-cd}.
Then
\begin{align}\label{intergral-equation}
\int_{\mathbb{B}} |f(z)|^p (1 - |z|)^\kappa \omega(z) \, dV(z) \simeq \int_{\mathbb{B}} |f(z)|^p (1 - |z|)^{\kappa-1} \hat{\omega}(z) \, dV(z), \quad f \in \mathcal{H}(\mathbb{B}).
\end{align}
\end{lemma}

\begin{proof}
The function $(1 - |\cdot|)^{\kappa-1} \hat{\omega}$ is a weight for each $\kappa>-\alpha$ by Lemma \ref{check-d-eq-cd}.
For the left-hand side of \eqref{intergral-equation}, expressing the integral in polar coordinates and applying integration by parts shows that 
\begin{align*}
\int_{\mathbb{B}} |f(z)|^p (1 - |z|)^\kappa \omega(z) \, dV(z)
&\simeq \int_{0}^{1} r^{2n-1} dr\int_{\mathbb{S}} \left|f(r\zeta)\right|^{p}(1-r)^{\kappa}\omega(r) \,d\sigma(\zeta)   \\
&=\int_{0}^{1} r^{2n-1}   (1-r)^{\kappa}\omega(r) M_{p}^{p}(r,f)\, dr  \\
&=-\int_{0}^{1} r^{2n-1} M_{p}^{p}(r,f) \, d\left( \int_{r}^{1} (1-t)^{\kappa} \omega(t) dt\right) \\
&=\int_0^1 \frac{\partial\left ( r^{2n-1} M_P^p(r,f) \right )}{\partial r}    \left( \int_r^1 (1 - t)^\kappa \omega(t) \, dt \right) dr, 
\end{align*}
where $M_{p}(r,f)=\left ( \int_{\mathbb{S}} \left|f(r\zeta)\right|^{p} d\sigma(\zeta) \right )^{\frac{1}{p} }$ and $d\sigma$ is the normalized area measure on $\mathbb{S}$.
Applying the analogous treatment to the right-hand side of \eqref{intergral-equation} yields
\begin{align*}
\int_{\mathbb{B}} |f(z)|^p (1 - |z|)^{\kappa-1} \hat{\omega}(z) \, dV(z)=\int_0^1 \frac{\partial\left ( r^{2n-1} M_P^p(r,f) \right )}{\partial r}    \left( \int_r^1 (1 - t)^{\kappa-1} \hat{\omega}(t) \, dt \right) dr.
\end{align*}
Now, we just need to prove
\begin{align*}
\int_0^1 \frac{\partial\left ( r^{2n-1} M_P^p(r,f) \right )}{\partial r}    \left( \int_r^1 (1 - t)^\kappa \omega(t) \, dt \right) dr 
\simeq \int_0^1 \frac{\partial\left ( r^{2n-1} M_P^p(r,f) \right )}{\partial r} \left( \int_r^1 (1 - t)^{\kappa - 1} \hat{\omega}(t) \, dt \right) dr,
\end{align*}
which can be simplifed to show that
\begin{align}
\int_r^1 (1 - t)^\kappa \omega(t) \, dt\simeq \int_r^1 (1 - t)^{\kappa - 1} \hat{\omega}(t) \, dt.
\end{align}
From one hand, since $\omega\in\hat{\mathcal{D}}$, according to integration by parts and Lemma \ref{hat-d-eq-cd}, we get
\begin{align*}
\int_r^1 (1 - t)^\kappa \omega(t) \, dt
&=-\int_r^1 (1 - t)^\kappa \, d\hat{\omega}(t) \\
&=(1-r)^\kappa \hat{\omega}(r)-\kappa \int_r^1 (1 - t)^{\kappa-1} \hat{\omega}(t)dt\\
&\lesssim (1-r)^\kappa \hat{\omega}(r)-\frac{\kappa\hat{\omega}(r)}{(1-r)^\beta } \int_r^1 (1 - t)^{\kappa+\beta-1} dt \\
&=(1-r)^\kappa \hat{\omega}(r)-\frac{\kappa\hat{\omega}(r)}{(1-r)^\beta } \frac{(1-r)^{\kappa+\beta}}{\kappa+\beta} \\
&\simeq (1-r)^\kappa \hat{\omega}(r), 
\end{align*}
and
\begin{align*}
\int_r^1 (1 - t)^{\kappa-1} \hat{\omega} (t) \, dt 
&\gtrsim \frac{\hat{\omega} (r)}{(1-r)^{\beta}}\int_r^1 (1 - t)^{\kappa+\beta-1}  \, dt \\
&=\frac{\hat{\omega} (r)}{(1 - r)^{\beta}} \frac{(1 - r)^{\kappa+\beta }}{\kappa+\beta}  \\
&\simeq (1 - r)^{\kappa} \hat{\omega} (r).
\end{align*}
Combining Lemma \ref{hat-d-eq-cd} with Lemma \ref{hankel-bb-lem4}, we obtain
\begin{align*}
	\int_r^1 (1 - t)^\kappa \omega(t) \, dt
	&\gtrsim\frac{\hat{\omega} (r)}{(1-r)^{\beta}}\int_r^1 \frac{(1 - t)^{\kappa+\beta}\omega(t)}{\hat{\omega}(t)}  dt \\
	&\simeq \frac{\hat{\omega} (r)}{(1-r)^{\beta}}(1 - r)^{\kappa+\beta} \\
	&=(1 - r)^{\kappa} \hat{\omega} (r).
\end{align*}
From the other hand, since $\omega\in\check{\mathcal{D}}$, Lemma \ref{check-d-eq-cd} implies that
\begin{align*}
\int_r^1 (1 - t)^{\kappa-1} \hat{\omega} (t) \, dt 
&\lesssim \frac{\hat{\omega} (r)}{(1-r)^{\alpha}}\int_r^1 (1 - t)^{\kappa+\alpha-1}  \, dt \\
&=\frac{\hat{\omega} (r)}{(1-r)^{\alpha}}\frac{(1 - r)^{\kappa+\alpha}}{\kappa+\alpha} 
\\
&=(1 - r)^{\kappa} \hat{\omega} (r).
\end{align*}
The desired conclusion is therefore immediately obtained.
\end{proof}

\begin{Proposition}\label{lem:schur}  
Let \(0 < q < \infty\), $\frac{1}{q}+\frac{1}{q'}=1$ and \(\omega \in \mathcal{D}\). Take $h=\hat{\omega}^{-\frac{1}{qq'}}$, then
\begin{align}\label{schur-2-q}
	\int_{\mathbb{B}} |B_{z}^{\omega}(\xi)| h(z)^{q} \omega(z) \, dV(z) \lesssim h(\xi)^{q}, \quad \xi \in \mathbb{B}
\end{align}
and
\begin{align}\label{schur-1-q'}
	\int_{\mathbb{B}} |B_{z}^{\omega}(\xi)| h(\xi)^{q'} \omega(\xi) \, dV(\xi) \lesssim h(z)^{q'}, \quad z \in \mathbb{B}.
\end{align}
\end{Proposition}

\begin{proof}
Take $h=\hat{\omega}^{-\frac{1}{qq'}}$, then 
\begin{align*}
\int_{t}^{1} h(s)^{q'} \omega(s) \, ds
&=\int_{t}^{1}\hat{\omega}(s)^{-\frac{1}{q}}\omega(s)\,ds  \\
&=-\int_{t}^{1}\hat{\omega}(s)^{-\frac{1}{q}}\,d\hat{\omega}(s) \\
&= \hat{\omega}(t)^{\frac{1}{q'}}
\end{align*}
for $0\leq t<1$. Therefore, the fact $\omega\in\check{\mathcal{D}}$ and Lemma \ref{check-d-eq-cd} yields 
\begin{equation}\label{schur-pre}
	\int_{0}^{r} \frac{\int_{t}^{1} h(s)^{q'} \omega(s) \, ds}{\hat{\omega}(t)(1 - t)} \, dt 
	\simeq \int_{0}^{r} \frac{dt}{\hat{\omega}(t)^{\frac{1}{q}} (1 - t)} 
	\lesssim \frac{1}{\hat{\omega}(r)^{\frac{1}{q}}} = h(r)^{q'},
	\quad 0 \leq r < 1.
\end{equation}
On the other hand, owing to \(\omega \in \hat{\mathcal{D}} \), we may Lemma \ref{DjLsLxSy2019-lem4} and asymptotic inequality \eqref{schur-pre} to deduce
\begin{align*}\label{schur-1-q'}
	\begin{aligned}	
		\int_{\mathbb{B}} |B_{z}^{\omega}(\xi)| h(\xi)^{q'} \omega(\xi) \, dV(\xi) 
		&\simeq \int_{0}^{1}s^{2n-1}ds\int_{\mathbb{S}}|B_{z}^{\omega}(s\eta)| h(s)^{q'} \omega(s)\,d\sigma(\eta) \\
		&\simeq \int_{0}^{1}s^{2n-1}h(s)^{q'} \omega(s) M_{1}^{1}\left(s,B_{z}^{\omega}\right) \, ds \\
		&\simeq \int_{0}^{1}s^{2n-1}h(s)^{q'} \omega(s) \,ds \int_{0}^{s|z|}\frac{1}{(1-t)\hat{\omega(t)}}\,dt  \\
		&\leq \int_{0}^{|z|} \frac{\int_{t}^{1} h(s)^{q'} \omega(s) \, ds}{\hat{\omega}(t)(1 - t)} \, dt \\
		&\lesssim h(z)^{q'}, \quad z \in \mathbb{B}.
	\end{aligned}
\end{align*}
As symmetry, a reasoning similar to that above yields \eqref{schur-2-q}.
\end{proof}

Auxiliary results will be needed to deduce our main result. 
The following theorem, proven in \cite{Djun2025TO}, is essential for our analysis.

\vspace{\baselineskip}

\noindent{\bf Theorem A. }{\it
Let $1 < p$, $q < \infty$, \(\omega \in \mathcal{D}\), and \(\mu\) be a positive Borel measure on \(\mathbb{B}\). 
The following assertions hold.
\begin{itemize}
\item[(I)]If \(1 < p \leq q < \infty\), then:
\begin{itemize}
\item[(i)]\(\mathcal{T}_\mu : A_\omega^p \to A_\omega^q\) is bounded if and only if 
\begin{align*}
\|T_\mu\|\simeq  \sup_{z \in \BB} \frac{\mu(S_{z})}{\omega(S_{z})^{1+\frac{1}{p}-\frac{1}{q}}} < \infty,
\end{align*}
which is equivalent to \(\mu\) being a \(\left(\frac{s}{p} - \frac{s}{q} + s\right)\)-Carleson measure for \(A_\omega^s\) for some (or equivalently, for all) \(0 < s < \infty\).
\item[(ii)]\(\mathcal{T}_\mu : A_\omega^p \to A_\omega^q\) is compact if and only if \begin{align*}
\lim_{|z|\to 1}\dfrac{\mu(S_z)}{\omega(S_z)^{1+\frac{1}{p}-\frac{1}{q}}}=0.
\end{align*}
\end{itemize}
\item[(II)]If \(1<q<p <\infty\), then \(\mathcal{T}_\mu : A_\omega^p \to A_\omega^q\) is bounded if and only if it is compact.
\end{itemize}}	
	
\vspace{\baselineskip}

A distinctive feature of the proof in \cite[Theorem 1.3]{Djun2025TO} is its reliance on Carleson squares, diverging from the traditional use of the (pseudo-)hyperbolic ball techniques in weighted Bergman space analysis. For any weight \(\omega \in \mathcal{R}\) and parameter \(\beta \in (0,1)\), there exists a constant \(C = C(\beta, \omega) > 0\) such that
\[
C^{-1} \omega\left(D\left(a, \beta(1 - |a|)\right)\right) \leq \omega\left(S_a\right) \leq C \, \omega\left(D\left(a, \beta(1 - |a|)\right)\right),\quad \forall \, a\in\BB.
\]
Hence, a \(\left(\frac{s}{p} - \frac{s}{q} + s\right)\)-Carleson measure for \(A_\omega^s\) can be equivalently characterized either Carleson blocks or (pseudo-)hyperbolic balls. Therefore, for \(\omega \in \mathcal{R}\), Theorem A (I) can be restated as follows:

\vspace{\baselineskip}

\noindent{\bf Theorem A (I) (Restated). }{\it
If \(1 < p \leq q < \infty\), \(\omega \in \mathcal{R}\) and \(\mu\) be a positive Borel measure on \(\mathbb{B}\). 
The following assertions are equivalent.
\begin{itemize}
\item[(i)]\(\mathcal{T}_\mu : A_\omega^p \to A_\omega^q\) is bounded.
\item[(ii)]$\|T_\mu\|\simeq  \sup_{z \in \BB} \frac{\mu(D(z,r))}{\omega(D(z,r))^{\frac{1}{p}-\frac{1}{q}+1}} < \infty.$
\item[(iii)]\(\mu\) is a \(\left(\frac{s}{p} - \frac{s}{q} + s\right)\)-Carleson measure for \(A_\omega^s\) for some (or equivalently, for all) \(0 < s < \infty\).
\end{itemize}}

\vspace{\baselineskip}

Furthermore, if \(\mathcal{T}_\mu : A_\omega^p \to A_\omega^q\) is bounded, then the equivalence
\begin{align}\label{eq-sup}
\sup_{z\in\mathbb{B}}\frac{\mu(S_z)}{\omega(S_z)^{\frac{1}{p}-\frac{1}{q}+1}}\simeq\sup_{z\in\mathbb{B}}\frac{\mu(D(z,r))}{\omega(D(z,r))^{\frac{1}{p}-\frac{1}{q}+1}}
\end{align}
holds for \(1 < p \leq q < \infty\) and \(\omega \in \mathcal{R}\).

Recall that the essential norm of a bounded linear operator $T:X \to Y$ is the distance from $T$ to the space of compact operators. Namely,
\begin{align*}
	\|T\|_{e} = \inf\left\{\|T - K\| : K \text{ is compact from } X \text{ to } Y\right\}.
\end{align*}
The essential norm of a linear operator \( T \), is intrinsically linked to its boundedness and compactness. Specifically, we know that \( \|T\|_e < +\infty \) if and only if \( T \) is bounded, and \( \|T\|_e = 0 \) if and only if \( T \) is compact. Consequently, estimating \( \|T\|_e \) provides a criteria for determining whether \( T \) is bounded or compact. The study of essential norms for operators on holomorphic function spaces is of significant importance in complex analysis and operator theory, see \cite{the-Essential-1987,weighted2004,Essential2007,the-Essential2015,Yang-2024} and the references therein for related works.

\section{Main results and proofs}

A careful inspection of the Theorem A (II)  in hand states that $T_{\mu}: A_{\omega}^{p} \to A_{\omega}^{q}$ is bounded if and only if it is compact whenever $1 < q < p < \infty$, and hence its essential norm is either \(0\) or \(+\infty\). Consequently, it suffices to consider the case $1<p\le q<\infty$.

With above preparations we can state the main result of this paper.
\begin{Theorem}\label{main-thm}
Let $1<p\le q<\infty$, $\om\in\mathcal{D}$ and $\mu$ be a positive Borel measure on $\BB$. If $\mathcal{T}_\mu:A_{\omega}^{p}\to A_{\omega}^{q}$ is bounded, then
\begin{align*}
\left\|\mathcal{T}_\mu\right\|_{e}\simeq\limsup_{|z|\to1}\frac{\mu(S_z)}{\omega(S_z)^{1+\frac{1}{p}-\frac{1}{q}}}.    
\end{align*}
\end{Theorem}

\begin{proof}[Proof of the lower estimate]
For fixed $a\in\BB$, we define functions $f_a(z)$ and $g_a(z)$ as following:
\begin{align*}
f_a(z)=\frac{B_a^\om(z)^{1+\frac{1}{p}-\frac{1}{q}}}{\left \| (B_a^\om)^{1+\frac{1}{p}-\frac{1}{q}} \right \|_{A_{\omega}^{p}}}
\mbox{\quad and \quad}
g_a(z)=\frac{B_a^\om(z)^{1+\frac{1}{p}-\frac{1}{q}}}{\left \| (B_a^\om)^{1+\frac{1}{p}-\frac{1}{q}} \right \|_{A_{\omega}^{q'}}}, 
\end{align*}
where $\frac{1}{q}+\frac{1}{q'}=1$. Note that $f_a\in A_{\omega}^{p}$ and $g_a\in A_{\omega}^{q'}$ with $\left\|f_a\right\|_{A_{\omega}^{p}}=\left\|g_a\right\|_{A_{\omega}^{q'}}=1$. 
Let $s=1+\frac{1}{p}-\frac{1}{q}$, then by Lemma \ref{TO-lem2-3} (ii),
\begin{align*}
\left\| (B_{a}^{\om})^{1+\frac{1}{p}-\frac{1}{q}} \right\|_{A_{\omega}^{p}}
= \left( \int_{\mathbb{B} } |B_{a}^{\om}(z)|^{\left ( 1+\frac{1}{p}-\frac{1}{q} \right )p } \omega(z) dV(z) \right)^{\frac{1}{p}}
=\left\| B_{a}^{\om} \right\|_{A_{\omega}^{sp}}^{s}
=\frac{1}{\omega(S_a)^{1-\frac{1}{q}}} . 
\end{align*}
Further, using Lemma \ref{TO-lem2-3} (i), for $|z|\le r<1$,
\begin{align*}
	|f_a(z)| \lesssim \omega(S_a)^{1-\frac{1}{q}}\|B^\omega_{ra}\|_{H^\infty}^{1+\frac{1}{p}-\frac{1}{q}}  \simeq \frac{\omega(S_a)^{1-\frac{1}{q}}}{\omega(S_{ra})^{1+\frac{1}{p}-\frac{1}{q}}}.
\end{align*}
This yields that $\left\{f_a\right\}$ is bounded in $A_{\omega}^{p}$ and converges to 0 uniformly on compact subsets of $\BB$ as $|a|\to1$. Let $K$ be an arbitrary compact operator from $A_{\omega}^{p}$ into $A_{\omega}^{q}$. Then $\left\|Kf_a\right\|_{A_{\omega}^{q}}\to 0$ as $|a|\to1$. By the Cauchy-Schwarz inequality,
\begin{align*}
\left|\int_{\mathbb{B} } Kf_{a}(z) \overline{g_{a}(z)} \omega(z) \,dV(z) \right|
= \left | \left \langle Kf_{a},g_{a} \right \rangle_{A_{\omega}^{2}}   \right | 
\leq \left \| Kf_{a} \right \| _{A_{\omega}^{q}} \left \| g_{a} \right \| _{A_{\omega}^{q'}}
\to 0 \quad \text{as }|a|\to1.
\end{align*}
Since $\mathcal{T}_\mu:A_{\omega}^{p}\to A_{\omega}^{q}$ is bounded, by the triangular inequality and H\"{o}lder's inequality, 
\begin{align*}
&\left|\int_{\mathbb{B} } \mathcal{T}_\mu f_{a}(z) \overline{g_{a}(z)} \omega(z) \,dV(z) \right|-\left|\int_{\mathbb{B} } Kf_{a}(z) \overline{g_{a}(z)} \omega(z) \,dV(z) \right| \\
\leq&\left|\int_{\mathbb{B} } (\mathcal{T}_\mu-K) f_{a}(z) \overline{g_{a}(z)} \omega(z) \,dV(z) \right|\\
\leq&\int_{\mathbb{B} } \left|(\mathcal{T}_\mu-K) f_{a}(z) \overline{g_{a}(z)} \omega(z)\right| dV(z) \\
\leq&\left( \int_{\mathbb{B} } \left|(\mathcal{T}_\mu-K) f_{a}(z)\right|^{q} \omega(z) \, dV(z) \right)^{\frac{1}{q} }  
\left( \int_{\mathbb{B} } \left|\overline{g_{a}(z)}\right|^{q'} \omega(z) \, dV(z) \right)^{\frac{1}{q'} }  \\
=&\left \| (\mathcal{T}_\mu-K) f_{a} \right \| _{A_{\omega}^{q}} \left \| g_{a} \right \| _{A_{\omega}^{q'}} \\
\leq& \left\| \mathcal{T}_\mu-K\right\|.
\end{align*}
Hence, 
\begin{align}\label{left-hand-TK}
\limsup_{|a|\to 1}\left|\int_{\mathbb{B} } \mathcal{T}_\mu f_{a}(z) \overline{g_{a}(z)} \omega(z) \,dV(z)\right|\leq \left\| \mathcal{T}_\mu-K\right\|.
\end{align}

Now, we deal with the left-hand side of the inequality \eqref{left-hand-TK}.
Using Fubini's theorem, and also the dominated converge theorem, we see that  
\begin{align*}
\int_{\mathbb{B} } \mathcal{T}_\mu f_{a}(z) \overline{g_{a}(z)} \omega(z) \,dV(z)
=& \int_{\mathbb{B}} \left( \int_{\mathbb{B}} f_a(\xi) \overline{B_z^{\om}(\xi)} \, d\mu(\xi) \right)  \overline{g_a(z)}  \om(z) \, dV(z) \\
=& \int_{\mathbb{B}} \left( \int_{\mathbb{B}} f_a(\xi) B_\xi^{\om}(z) \, d\mu(\xi) \right)  \overline{g_a(z)}  \om(z) \, dV(z) \\
=& \int_{\mathbb{B}} f_a(\xi) \, d\mu(\xi) \int_{\mathbb{B}} \overline{g_a(z)} B_\xi^{\om}(z)   \om(z) \, dV(z) \\
=& \int_{\mathbb{B}} f_a(\xi) \, d\mu(\xi) \overline{\int_{\mathbb{B}} g_a(z) \overline{B_\xi^{\om}(z)} \om(z) \, dV(z)}  \\
=&  \int_{\mathbb{B}} f_a(\xi) \overline{g_a(\xi)}\, d\mu(\xi)\\
=&\int_{\mathbb{B}} f_{a}(z) \overline{g_{a}(z)} \, d\mu(z).
\end{align*}
Consequently, using Lemma \ref{TO-lem2-3} (ii) again, we deduce
\begin{align*}
\left\| \mathcal{T}_\mu-K\right\| 
&\ge \limsup_{|a|\to 1}	\left | \int_{\mathbb{B}} f_{a}(z) \overline{g_{a}(z)} \, d\mu(z) \right |  \\
&\simeq \omega(S_{a})^{1+\frac{1}{p}-\frac{1}{q}} \int_{\mathbb{B} }\left | B_a^{\om}(z) \right |^{2\left ( 1+\frac{1}{p}-\frac{1}{q} \right ) }   \, d\mu(z).
\end{align*}
Let $\delta\in(0,1)$ be the constant in Lemma \ref{TO-lem2-3} (iii), then
\begin{align*}
\mu\left(S_{a_{\delta}}\right)=\int_{S_{a_{\delta}}}d\mu(z)
&\lesssim \omega(S_{a})^{2(1+\frac{1}{p}-\frac{1}{q})} \int_{S_{a_{\delta}}}\left | B_a^{\om}(z) \right |^{2\left ( 1+\frac{1}{p}-\frac{1}{q} \right ) }   \, d\mu(z)  \\
&\leq \omega(S_{a})^{2(1+\frac{1}{p}-\frac{1}{q})} \int_{\mathbb{B}}\left | B_a^{\om}(z) \right |^{2\left ( 1+\frac{1}{p}-\frac{1}{q} \right ) }   \, d\mu(z)
\end{align*}
for $a\in\mathbb{B} \setminus \{0\}$,  which together with Lemma \ref{lem:2.2} (iii) gives 
\begin{align*}
\frac{\mu\left(S_{a_{\delta}}\right)}{\omega(S_{a_{\delta}})^{1+\frac{1}{p}-\frac{1}{q}}} 
&\lesssim \omega(S_{a})^{1+\frac{1}{p}-\frac{1}{q}} \int_{\mathbb{B}}\left | B_a^{\om}(z) \right |^{2\left ( 1+\frac{1}{p}-\frac{1}{q} \right ) }   \, d\mu(z).  
\end{align*}
Bearing in mind Theorem A (I) and the boundness of $\mathcal{T}_\mu$, we may take the supremum of $|a_{\delta}|$ over the interval $(1-\delta,1)$, then
\begin{align*}
\sup_{1-\delta<|a_{\delta}|<1} \frac{\mu\left(S_{a_{\delta}}\right)}{\omega(S_{a_{\delta}})^{1+\frac{1}{p}-\frac{1}{q}}} \lesssim \omega(S_{a})^{1+\frac{1}{p}-\frac{1}{q}} \int_{\mathbb{B} }\left | B_a^{\om}(z) \right |^{2\left ( 1+\frac{1}{p}-\frac{1}{q} \right ) }   \, d\mu(z).
\end{align*}
Then we get
\begin{align*}
\limsup_{|a|\to1} \frac{\mu\left(S_{a}\right)}{\omega(S_{a})^{1+\frac{1}{p}-\frac{1}{q}}} \lesssim \limsup_{|a|\to1}\omega(S_{a})^{1+\frac{1}{p}-\frac{1}{q}} \int_{\mathbb{B} }\left | B_a^{\om}(z) \right |^{2\left ( 1+\frac{1}{p}-\frac{1}{q} \right ) }   \, d\mu(z) \lesssim \left\| \mathcal{T}_\mu-K\right\|.	
\end{align*}
Since the compact operator $K$ from $A_{\omega}^{p}$ into $A_{\omega}^{q}$ is arbitrary, it follows that
\begin{align*}
\limsup_{|a|\to1} \frac{\mu\left(S_{a}\right)}{\omega(S_{a})^{1+\frac{1}{p}-\frac{1}{q}}}\lesssim\left\|\mathcal{T}_\mu\right\|_{e}.
\end{align*}
This finishes the proof of the lower estimate.
\end{proof}

\begin{proof}[Proof of the upper estimate]
For any fixed $0<t<1$ and $f\in A_{\omega}^{p}$, split the integral $\mathcal{T}_\mu f$ into two parts:
\begin{align*}
\mathcal{T}_\mu f(z) &= \int_{\mathbb{B}} f(\xi) \overline{B_z^\omega(\xi)} \, d\mu(\xi)\\
&=\int_{t\mathbb{B}} f(\xi) \overline{B_z^\omega(\xi)} \, d\mu(\xi)
+\int_{\mathbb{B}-t\mathbb{B}} f(\xi) \overline{B_z^\omega(\xi)} \, d\mu(\xi) \\
&=T_1f(z)+T_2f(z),	
\end{align*}
where
\begin{align*}
&T_1f(z)=\int_{t\mathbb{B}} f(\xi) \overline{B_z^\omega(\xi)} \, d\mu(\xi),\\
&T_2f(z)=\int_{\mathbb{B}-t\mathbb{B}} f(\xi) \overline{B_z^\omega(\xi)} \, d\mu(\xi).
\end{align*}
We first claim that $T_1:A_{\omega}^{p}\to A_{\omega}^{q}$ is compact. By Lemma \ref{lem:2.4}, it suffices to prove that for arbitrary uniformly bounded sequence $\left \{ f_k \right \}\subseteq A_{\omega}^{p}$ converging to 0 uniformly on compact subsets of $\mathbb{B}$, $\left \| T_1f_k \right \|_{A_{\omega}^{q}}\to 0$ as $k\to\infty$. With Lemma \ref{TO-lem2-3} (i), one readily sees that 
\begin{align*}
\left|T_1f_k(z)\right| &\le \int_{t\mathbb{B}} \left | f_k(\xi) \overline{B_z^\omega(\xi)} \right |  \, d\mu(\xi) \\
&\le \mu(\mathbb{B})\left \| B_{tz}^\omega\right \|_{H^{\infty}} \sup_{|\xi|<t}|f_k(\xi)| \\
&\lesssim \frac{\mu(\mathbb{B})}{\omega (S_{tz} )}  \sup_{|\xi|<t}|f_k(\xi)|.
\end{align*}
This yields $\left \| T_1f_k \right \|_{A_{\omega}^{q}}\to 0$ as $k\to\infty$, and then $T_1$ should be compact.

We next consider $T_2$. By the fact $\chi_{D(\xi,r)}(u)=\chi_{D(u,r)}(\xi)$, applying \cite[Lemma 2.24]{Zk2005} and Fubini's theorem, we deduce
\begin{align*}
\left|T_2f(z)\right|&\le\int_{\mathbb{B}-t\mathbb{B}} \left | f(\xi) \overline{B_z^\omega(\xi)} \right |  \, d\mu(\xi) \\
&\lesssim\int_{\mathbb{B}-t\mathbb{B}} \frac{1}{\left ( 1-|\xi| \right )^{n+1}  }\int_{D(\xi,r)} \left | f(u) \overline{B_z^\omega(u)} \right |  \,dV(u) d\mu(\xi) \\
&=\int_{\mathbb{B}-t\mathbb{B}} \frac{1}{\left ( 1-|\xi| \right )^{n+1}  }\int_{\mathbb{B}}  \chi _{D(\xi,r)}(u)   \left | f(u) \overline{B_z^\omega(u)} \right |  \,dV(u) d\mu(\xi)\\
&=\int_{\mathbb{B}}\left | f(u) \overline{B_z^\omega(u)} \right |  \,dV(u) \int_{\mathbb{B}-t\mathbb{B}}  \frac{\chi _{D(u,r)}(\xi)}{\left ( 1-|\xi| \right )^{n+1}  }   d\mu(\xi).\\
\end{align*}
Note that $\beta(u,\xi)<r$, then $1-|u|\simeq1-|\xi|$ by \cite[Lemma 2.20]{Zk2005} and hence $\hat{\omega}(u)\simeq\hat{\omega}(\xi)$ by Lemma \ref{lem:2.2} (iii). This together with Lemma \ref{lem:2.2} (ii) gives
\begin{align*}
\left|T_2f(z)\right|	
&\lesssim \int_{\mathbb{B}}\left | f(u) \overline{B_z^\omega(u)}\right| \frac{\left ( 1-|u| \right )^{-1+\left(\frac{1}{p}-\frac{1}{q}\right)n } \hat{\omega}(u)^{1+\frac{1}{p}-\frac{1}{q}}}{\left ( 1-|u| \right )^{-1+\left(\frac{1}{p}-\frac{1}{q}\right)n } \hat{\omega}(u)^{1+\frac{1}{p}-\frac{1}{q}}}    \,dV(u) 
\int_{\mathbb{B}-t\mathbb{B}}  \frac{\chi _{D(u,r)}(\xi)}{\left ( 1-|\xi| \right )^{n+1}  }   d\mu(\xi) \\
&\simeq \int_{\mathbb{B}}\left | f(u) \overline{B_z^\omega(u)}\right| \left ( 1-|u| \right )^{-1+\left(\frac{1}{p}-\frac{1}{q}\right)n } \hat{\omega}(u)^{1+\frac{1}{p}-\frac{1}{q}}   \,dV(u) 
\int_{\mathbb{B}-t\mathbb{B}}  \frac{\chi _{D(u,r)}(\xi)}{\left[ \left(1-|\xi|\right)^{n} \hat{\omega} (\xi) \right]^{1+\frac{1}{p}-\frac{1}{q}}}   d\mu(\xi) \\
&\simeq \int_{\mathbb{B}}\left | f(u) \overline{B_z^\omega(u)}\right| \left ( 1-|u| \right )^{-1+\left(\frac{1}{p}-\frac{1}{q}\right)n } \hat{\omega}(u)^{1+\frac{1}{p}-\frac{1}{q}}   \,dV(u) \int_{\mathbb{B}-t\mathbb{B}}  \frac{\chi _{D(u,r)}(\xi)}{\omega(S_{\xi})^{1+\frac{1}{p}-\frac{1}{q}}}   d\mu(\xi) \\
&\leq \sup_{\substack{u\in D(\xi,r) \\ \xi\in\mathbb{B}-t\mathbb{B}}} \frac{\mu\left ( D(u,r) \right ) }{\omega(S_{u})^{1+\frac{1}{p}-\frac{1}{q}}} \int_{\mathbb{B}}\left | f(u) \overline{B_z^\omega(u)}\right| \left ( 1-|u| \right )^{-1+\left(\frac{1}{p}-\frac{1}{q}\right)n } \hat{\omega}(u)^{1+\frac{1}{p}-\frac{1}{q}}   \,dV(u)\\
&\simeq \sup_{\substack{u\in D(\xi,r) \\ \xi\in\mathbb{B}-t\mathbb{B}}} \frac{\mu\left ( D(u,r) \right ) }{\omega(S_{u})^{1+\frac{1}{p}-\frac{1}{q}}}
\int_{\mathbb{B}}\left | f(u) \overline{B_z^\omega(u)}\right| \left ( 1-|u| \right )^{\left(\frac{1}{p}-\frac{1}{q}\right)(n+1) } \omega(u)^{1+\frac{1}{p}-\frac{1}{q}}  \,dV(u),
\end{align*}
where the last asymptotic equality follows by employing Lemma \ref{intergral-eq}.
Set
\begin{align*}
\mathcal{T}f(z):=\int_{\mathbb{B}}\left | f(u) \overline{B_z^\omega(u)}\right| \left ( 1-|u| \right )^{\left(\frac{1}{p}-\frac{1}{q}\right)(n+1) } \omega(u)^{1+\frac{1}{p}-\frac{1}{q}}  \,dV(u).
\end{align*}
Then
\begin{align}\label{T2-T-f}
\left\|T_2f\right\|_{A_{\omega}^{q}}\lesssim\sup_{\substack{u\in D(\xi,r) \\ \xi\in\mathbb{B}-t\mathbb{B}}} \frac{\mu\left ( D(u,r) \right ) }{\omega(S_{u})^{1+\frac{1}{p}-\frac{1}{q}}}\left\|\mathcal{T}f\right\|_{L_{\omega}^{q}}
\end{align}
and we proceed to estimate the norm $\left\|\mathcal{T}f\right\|_{L_{\omega}^{q}}$ for $f\in A_{\omega}^{p}$.

Now, H\"{o}lder's inequality together with \eqref{schur-1-q'} in Proposition \ref{lem:schur} show that 
\begin{align*}
\mathcal{T}f(z)
=&\int_{\mathbb{B}}\frac{\left | f(u) \overline{B_z^\omega(u)}\right|h(u)}{h(u)}  \left ( 1-|u| \right )^{\left(\frac{1}{p}-\frac{1}{q}\right)(n+1) } \omega(u)^{1+\frac{1}{p}-\frac{1}{q}}  \,dV(u) \\ 
\leq &\left( \int_{\mathbb{B}} \left|B_z^\omega(u)\right| h(u)^{q'} \omega(u) \,dV(u) \right)^{\frac{1}{q'}}  \left ( \int_{\mathbb{B}} \frac{ \left| B_z^{\omega}(u) \right| |f(u)|^{q} ( 1-|u|)^{\left(\frac{q}{p}-1\right)(n+1) } \omega(u)^{\frac{q}{p}}}{h(u)^{q} }  \,dV(u)
\right )^{\frac{1}{q}} \\
\lesssim & h(z) \left ( \int_{\mathbb{B}} \frac{ \left| B_z^{\omega}(u) \right| |f(u)|^{q} ( 1-|u|)^{\left(\frac{q}{p}-1\right)(n+1) } \omega(u)^{\frac{q}{p}}}{h(u)^{q} }  \,dV(u)
\right )^{\frac{1}{q}}.
\end{align*}	
Further, Fubini's theorem and \eqref{schur-2-q} in Proposition \ref{lem:schur} give
\begin{align*}
\left\|\mathcal{T}f\right\|_{L_{\omega}^{q}}^{q}=&\int_{\mathbb{B}}\left| \mathcal{T}f(z) \right|^{q}\omega(z)\,dV(z)   \\
\lesssim & \int_{\mathbb{B}} h(z)^{q} \omega(z)\,dV(z)
\int_{\mathbb{B}} \frac{ \left| B_z^{\omega}(u) \right| |f(u)|^{q} ( 1-|u|)^{\left(\frac{q}{p}-1\right)(n+1) } \omega(u)^{\frac{q}{p}}}{h(u)^{q} } \,dV(u)\\
=&\int_{\mathbb{B}} \frac{ |f(u)|^{q} ( 1-|u|)^{\left(\frac{q}{p}-1\right)(n+1) } \omega(u)^{\frac{q}{p}}}{h(u)^{q} } \,dV(u)
\int_{\mathbb{B}} \left|B_z^{\omega}(u)\right| h(z)^{q} \omega(z)\,dV(z)\\
\lesssim&\int_{\mathbb{B}} |f(u)|^{q} ( 1-|u|)^{\left(\frac{q}{p}-1\right)(n+1) } \omega(u)^{\frac{q}{p}} \,dV(u).
\end{align*}

Let $W(t)=\frac{\hat{\omega}(t)}{1-t}$, then $\hat{W}=\hat{\omega}$, $W\in\mathcal{R}$ and $\left\|\cdot\right\|_{A_{W}^{p}}\simeq\left\|\cdot\right\|_{A_{\omega}^{p}}$ for $0<p<\infty$.  
The subharmonicity of $f$ follows
\begin{align*}
\left\|\mathcal{T}f\right\|_{L_{\omega}^{q}}^{q}
\lesssim &\int_{\mathbb{B}}  
\left ( \frac{1}{W\left( D(u,r) \right)} \int_{D(u,r)} |f(\zeta)|^{p} W(\zeta)\, dV(\zeta) \right )^{\frac{q}{p}}   
( 1-|u|)^{\left(\frac{q}{p}-1\right)(n+1) } \omega(u)^{\frac{q}{p}} \,dV(u).
\end{align*}
Using the fact that $W\in\mathcal{R}$, $\hat{W}=\hat{\omega}$ and \cite[Proposition 3.1]{Djun2025TO}, we get 
$$W\left( D(u,r) \right)\simeq W\left( S_u \right)\simeq(1 - |u|)^{n} \hat{W}(u)
=(1 - |u|)^{n} \hat{\omega}(u).$$
This observation, combined with Lemma \ref{intergral-eq} shows that
\begin{align*}
\left\|\mathcal{T}f\right\|_{L_{\omega}^{q}}^{q}
\lesssim& \int_{\mathbb{B}}  
\left ( \frac{1}{(1 - |u|)^{n} \hat{\omega}(u)} \int_{D(u,r)} |f(\zeta)|^{p} W(\zeta)\, dV(\zeta) \right )^{\frac{q}{p}}   
( 1-|u|)^{\left(\frac{q}{p}-1\right)(n+1) } \omega(u)^{\frac{q}{p}} \,dV(u)\\
\simeq & \int_{\mathbb{B}}  \frac{1}{( 1-|u|)^{n+1}} 
\left (  \int_{D(u,r)} |f(\zeta)|^{p} W(\zeta)\, dV(\zeta) \right )^{\frac{q}{p}}  \,dV(u)\\
= & \int_{\mathbb{B}}  \frac{1}{( 1-|u|)^{n+1}} 
\left (  \int_{\mathbb{B}} \chi_{D(u,r)}(\zeta)  |f(\zeta)|^{p} W(\zeta)\, dV(\zeta) \right )^{\frac{q}{p}}  \,dV(u).\\
\end{align*}
Applying Minkowski's integral inequality to the right-hand side and recognizing that $\chi_{D(u,r)}(\zeta)=\chi_{D(\zeta,r)}(u)$, the above formula becomes
\begin{align*}
\left\|\mathcal{T}f\right\|_{L_{\omega}^{q}}^{q}
\lesssim & \left\{ \int_{\mathbb{B}} |f(\zeta)|^{p} W(\zeta)\, dV(\zeta)
\left( \int_{\mathbb{B}} \frac{\chi_{D(\zeta,r)}(u)}{( 1-|u|)^{n+1}} \,dV(u) \right)^{\frac{p}{q}}   \right\}^{\frac{q}{p}}  \\
\simeq& \left( \int_{\mathbb{B}} |f(\zeta)|^{p} W(\zeta)\, dV(\zeta) \right)^{\frac{q}{p}}\\
=&\left\|f\right\|_{A_{W}^{p}}^{q} \simeq\left\|f\right\|_{A_{\omega}^{p}}^{q}. 
\end{align*}
From the estimate \eqref{T2-T-f}, the operator norm of $T_2$ satisfies:
\begin{align*}
\left\|T_2\right\| \lesssim\sup_{\substack{u\in D(\xi,r) \\ \xi\in\mathbb{B}-t\mathbb{B}}} \frac{\mu\left ( D(u,r) \right ) }{\omega(S_{u})^{1+\frac{1}{p}-\frac{1}{q}}}.
\end{align*}
By the definition of the essential norm and subadditivity, we deduce
\begin{align}\label{upper-T}
\left\|\mathcal{T}_\mu\right\|_{e}=\left\|T_1+T_2\right\|_{e}\le\left\|T_2\right\| &\lesssim\sup_{\substack{u\in D(\xi,r) \\ \xi\in\mathbb{B}-t\mathbb{B}}} \frac{\mu\left ( D(u,r) \right ) }{\omega(S_{u})^{1+\frac{1}{p}-\frac{1}{q}}}.   
\end{align}
Taking $t\to1$, we further refine the estimate to 
\begin{align}\label{upper-T-last}
\left\|\mathcal{T}_\mu\right\|_{e}
&\lesssim \limsup_{|u|\to 1^{-}}  \frac{\mu\left ( D(u,r) \right ) }{\omega(S_{u})^{1+\frac{1}{p}-\frac{1}{q}}} 
\simeq \limsup_{|u|\to 1^{-}} \frac{\mu(S_u)}{\omega(S_u)^{1+\frac{1}{p}-\frac{1}{q}}}.
\end{align}
The last asymptotic equality in \eqref{upper-T-last} follows from Lemma \ref{hat-d-eq-cd} and Lemma \ref{lem:2.2} (i), which asserts that $\omega(S_{u})\simeq(1 - |u|)^{n}\hat{\omega}(u)\in \mathcal{R}$ as \(|u| \to 1\).
This completes the proof of the upper estimate. 
\end{proof}

\begin{Remark}
	It is noteworthy that when $p=q=2$, the condition $\om\in\mathcal{D}$ can be relaxed to $\om\in\hat{\mathcal{D}}$ and the conclusion in Theorem \ref{main-thm} still holds. This corresponds to the special case $k=0$ of Theorem 4 in \cite{DjLsDQ-2022}.
\end{Remark}

\end{document}